\RequirePackage{fix-cm}
\documentclass[5pt]{article}
\usepackage{amsfonts}
\usepackage{mathrsfs}
\usepackage{amsmath}
\usepackage{amsthm}
\usepackage{enumerate}
\usepackage{multicol}
\usepackage{lipsum}
\makeatletter
\def\ps@pprintTitle{%
   \let\@oddhead\@empty
   \let\@evenhead\@empty
   \def\@oddfoot{\reset@font\hfil\thepage\hfil}
   \let\@evenfoot\@oddfoot
}

\usepackage{graphicx}

\usepackage[top=3cm, bottom=3cm, left=3cm, right=2cm] {geometry}

\usepackage{hyperref}
\usepackage{fontenc}
\usepackage[v2,cmtip]{xy}
\usepackage{amsmath}
\usepackage{bm}
\usepackage{amssymb}
\usepackage{mathrsfs}
\usepackage{latexsym}
\usepackage{exscale}
\usepackage{eucal}
\vspace{2ex}

\title{\flushleft { Sieve methods and the twin prime conjecture  }} \vspace{2ex}
\author{\normalsize  Mbakiso Fix Mothebe \\
\normalsize \it Department of Mathematics, University of Botswana, Pvt Bag 00704, \\
 \it Gaborone, Botswana
 \\
  \normalsize {\it e-mail add}: mothebemf@ub.ac.bw }
\date{}
\theoremstyle{plain}
\newtheorem{thm}{Theorem}[section]
\newtheorem{theorem}[thm]{Theorem}
\newtheorem{proposition}[thm]{Proposition}
\newtheorem{lemma}[thm]{Lemma}
\newtheorem{conjecture}[thm]{Conjecture}
\newtheorem{corollary}[thm]{Corollary}
\theoremstyle{definition}

\newtheorem{example}[thm]{Example}

\theoremstyle{plain}

\theoremstyle{definition}

\def\leq{\leqslant}
\def\geq{\geqslant}
\def\DD{D\kern-.7em\raise0.4ex\hbox{\char '55}\kern.33em}

\newtheorem*{acknow}{Acknowledgment}

\makeatletter
\def\blfootnote{\xdef\@thefnmark{}\@footnotetext}
\makeatother

\begin{document}
\maketitle
\pagestyle{plain}


\begin{abstract}
For $n \geq 3,$ let $ p_n $ denote the $n^{\rm th}$ prime number. Let $[ \; ]$ denote the floor or greatest integer function. For a positive integer  $m,$ let $\pi_2(m)$ denote the number of twin primes not exceeding $m.$ The twin prime conjecture states that there are infinitely many prime numbers $p$ such that $p+2$ is also prime. In this paper we state a conjecture to the effect that given any integer $a>0$ there exists an integer $N_2(a)$ such that
 $$ \left[\frac{ap^2_{n+1}}{2(n+1)} \right] \leq \pi_2\left(p^2_{n+1} \right) $$ for all $n \geq N_2(a)$
 and prove the conjecture in the case $a=1.$ This, in turn, establishes the twin prime conjecture.
\end{abstract}



\section{Introduction and main results}\label{s1}

An integer $p \geq 2$ is called a prime if its only positive
divisors are $1$ and $p.$ The prime numbers form a sequence:
\begin{equation}  \label{primes}
2,\; 3,\;5, \;7, \; 11, \; 13, \; 17, \; 19, \; 23, \; 29, \; 31, \;
37, \; 41,\; 43,\;47 \; \ldots.
\end{equation}
Euclid (300 B.C.) considered prime numbers and proved that there are infinitely
many.
Prime numbers are odd except $2$ and the only consecutive prime numbers are $%
2$ and $3.$ Any two odd prime numbers in the sequences
(\ref{primes}) differ by at least $2.$ Pairs of prime numbers that
differ by $2$ as, for example, in the sequence below
\begin{equation}  \label{twinp}
(3,5), \; (5,7), \;(11,13), \;(17,19), \;(29,31), \;(41,43), \;
\ldots
\end{equation}
are said to be \textit{twin primes.}

\begin{conjecture}\label{gt}
 There exist infinitely many twin primes.
\end{conjecture}

The conjecture is still open. The first known published reference to this question was made by Alphonse de Polignac in $1849,$ who conjectured that for
 every even number $k,$ there are infinitely many pairs of prime numbers $p$ and $p'$ such that $p'-p = k$ (see \cite{Po}). The case $k=2$ is the twin prime conjecture. The conjecture has not yet been proven or unproven for a given value of $k.$ In $2013,$ an important breakthrough was made by Yitang Zhang who proved the conjecture for some value of $k< 70 \; 000 \; 000$ (see \cite{zhang}). Later that same year, James Maynard announced a related breakthrough which proved the conjecture for some $k< 600$ (see \cite{mayn}). In 2014 D.H.J. Polymath proved the conjecture for some $k \leq 246.$ (see \cite{poly})

In this paper, we prove that  Conjecture \ref{gt} is true.

 Let $[ \; ]$  denote the floor or greatest integer function and let $\pi_2(m)$ denote the number of twin primes not exceeding the positive integer $m.$ Our conjecture is the following:

\vspace{2ex}
\noindent \textbf{AMS Subject Classification:} 11N05; 11N36.

\vspace{.08in} \noindent \textbf{Keywords}: Primes, Twin primes,
Sieve methods.
\vspace{2ex}

\begin{conjecture}\label{maint4}  For $n \geq 3,$ let $ p_n $ denote the $n^{\rm th}$ prime.  Then for each integer $a>0$ there exists an integer $N_2(a)$ such that
 $$ \left[\frac{ap^2_{n+1}}{2(n+1)} \right] \leq \pi_2\left(p^2_{n+1} \right) $$ for all $n \geq N_2(a).$
\end{conjecture}
   The following is our main result:
\begin{theorem}\label{maint33}  Conjecture \ref{maint4} is true in the case $a=1.$
\end{theorem}
Thus our main result is that there exists an integer $N_2(1)$ such that
 $$ \left[\frac{p^2_{n+1}}{2(n+1)} \right] \leq \pi_2\left(p^2_{n+1} \right) $$
 for all $n \geq N_2(1).$ We shall show later that $N_2(1) = 20.$

 The sequence $  \left[\frac{p^2_{n+1}}{2(n+1)} \right] $ is unbounded. As a consequence,  Conjecture \ref{gt} holds if there are infinitely many integers $n$ for which $ \left[\frac{p^2_{n+1}}{2(n+1)} \right] \leq \pi_2\left(p^2_{n+1} \right) .$

 In \cite{moth} we claim to prove the following result which also implies the twin prime conjecture:
 \begin{theorem}\label{maint133}  For $n \geq 2,$ let $ p_n $ denote the $n^{\rm th}$ prime.  Then
 $$ \left[\frac{p^2_{n+3}}{3(n+2)} \right] \leq \pi_2\left(p^2_{n+3} \right)$$ for all $n \geq 2.$
\end{theorem}
Our attempt to prove the above result has, however, not been deemed sufficiently clear. In this paper we use a similar approach as in \cite{moth} to derive our main result, namely Theorem \ref{maint33}. The result of Theorem \ref{maint33} supersedes that of Theorem \ref{maint133} and therefore establishes its validity. Theorem \ref{maint33} proves a stronger case of Conjecture \ref{gt}, namely that the number of twin primes between $ p_n $ and $ p^2_{n+1} $ is unbounded with $n$ or as $n$ increases. Several cases of Conjecture \ref{maint4} may be deduced from our arguments in this paper.

Our work is organized as follows: In Section \ref{prelim}, we recall the definition of the well known sieve of Eratosthenes and state some preliminary results. Finally, a proof for Theorem \ref{maint33} is presented in Section \ref{proof2}.

  The usage of the word sieve in some cases of this work refers to subtracting magnitudes, as opposed to the usual reference to sifting out of integer positions.
  One of the basic principles applied in our methods of proof is that if two uneven sieves are individually applied to two finite sets of the same order, then after some finite time the order of the residue of the more porous one will be less that that of the other sieve. This principle yields our main results in this paper.

 \section{Preliminary Results}\label{prelim}

The concepts required are
elementary and can be obtained from introductory texts on number
theory, discrete mathematics and set theory (\cite{Burt},\cite{Edgar}, \cite{strayer}).

Eratosthenes ($276-194$ B.C.) was a Greek mathematician whose work in
 number theory remains significant. Consider the following lemma:

 \begin{lemma}\label{t1}
 Let $a > 1$ be an integer. If $a$ is not divisible by a prime
 number $p \leq \sqrt{a},$ then a is a prime.
 \end{lemma}

 Eratosthenes used the above lemma as a
 basis of a technique
  called ``Sieve of Eratosthenes" for finding all the prime numbers less than a given
   integer $x$.
   The algorithm calls for writing down the integers from $2$ to $x$
   in their natural order. The composite numbers in the sequence are then sifted out by crossing off from $2,$ every second number (all multiples of two) in the list, from the next remaining number, $3,$ every third number, from the next remaining number, $5,$ every fifth number, and so on for all the  remaining prime numbers less than or equal to
   $ \sqrt{x}.$ The integers that are left on the list  are primes. We shall refer to the set
   of integers left as the {\bf residue} of the sieve. Thus the order of the residue set is equal to $\pi(x),$ the number of primes not exceeding the integer  $x.$ In our application of the sieve of Eratosthenes the prime numbers $2 , 3,5, \ldots , p$ are also sifted out from the sequence so that if $p$ is the $n^{\rm th}$ prime, then the residue has order $\pi(x) -n.$

   We shall require the following results of J.B. Rosser and L. Schoenfeld (see \cite{ross} page 69):
\begin{theorem}\label{ross1} {\rm(J.B.Rosser,  L. Schoenfeld)} Let $n \geq 1$ be an integer. Then:
\begin{itemize}
\item[{\rm (i)}] $\frac{n}{(\log n -\frac{1}{2})} < \pi(n)$ \; \; for \; $n \geq 67, $
 \item[{\rm (ii)}] $ \pi(n) < \frac{n}{(\log n -\frac{3}{2})} $ \; \; for $n \geq e^{\frac{3}{2}}.$
\end{itemize}
\end{theorem}
\begin{corollary}\label{ross2}  Let $n \geq 1$ be an integer. Then:
\begin{itemize}
 \item[ ] $\frac{n}{\log n} < \pi(n) \; \;$ for \; $n \geq 17.$
 \end{itemize}
\end{corollary}
\begin{theorem}\label{ross3}  {\rm (J.B.Rosser, L. Schoenfeld)} Let $n \geq 1$ be an integer. Then:
\begin{itemize}
\item[{\rm (i)}] $n(\log n +\log \log n - {\frac{3}{2}}) < p_n$ \; \; for \; $n \geq 2,$
 \item[{\rm (ii)}] $  p_n < n(\log n +\log \log n - {\frac{1}{2}})$ \; \; for \; $n \geq 20.$
\end{itemize}
\end{theorem}
\begin{corollary}\label{ross4}  Let $n \geq 1$ be an integer. Then:
 \begin{itemize}
\item[{\rm (i)}] $n(\log n ) < p_n$ \; \; for \; $n \geq 1,$
 \item[{\rm (ii)}] $  p_n < n(\log n +\log \log n )$ \; \; for \; $n \geq 6.$
\end{itemize}
\end{corollary}
As a consequence of the above results we have the following result:
\begin{theorem}\label{maint32}
For $n \geq 3,$ let $p_n$ denote the $n^{\rm th}$ prime. Then for each integer $b > 0$  there exists an integer $N(b)$ such that
$$ \frac {b p^2_{n+1}}{n+1} <\pi\left(p^2_{n+1} \right)$$ for all $ n \geq N(b).$
\end{theorem}
\begin{proof}
Let $b > 0$ be an integer. By Corollary \ref{ross2},
  $\pi(x)>\frac{x}{\log x}$ for $ x>17.$ It therefore suffices to show that there exists an integer $N(b)$ such that
 $$\frac{ p_{n+1}^2}{2{\log}(p_{n+1})}>\frac{b p_{n+1}^2}{n+1}$$ for all $n \geq N(b)$ or, equivalently, such that $\log p_{n+1}<\frac{(n+1)}{2b}.$
  By Corollary \ref{ross4} (ii),
$$p_{n+1}<(n+1)({\log}(n+1)+{\log}{\log}(n+1))$$ for $n>6.$
 Thus, it is enough to find $N(b)$ such that
$${\rm log}((n+1)({\log}(n+1)+{\log}{\log}(n+1)))<\frac{(n+1)}{2b}.$$ This is always possible to achieve since
   if we treat ${\log}((n+1)({\log}(n+1)+{\rm log}{\rm log}(n+1)))$ as a function of $n,$ we get its derivative to be less than $\frac{3}{(n + 1)},$ which is smaller than $ \frac{1}{2b}$ for $n$ large enough. Thus for values of $n$ large enough  ${\log}((n+1)({\log}(n+1)+{\log}{\log}(n+1)))$  is less than $\frac{(n+1)}{2b}.$
 \end{proof}
 For example, $ \frac {2 p^2_{n+1}}{n+1} <\pi\left(p^2_{n+1} \right)$ for all $ n \geq 12,$ $ \frac {3 p^2_{n+1}}{n+1} <\pi\left(p^2_{n+1} \right)$ for all $ n \geq 23,$ $ \frac {4 p^2_{n+1}}{n+1} <\pi\left(p^2_{n+1} \right)$ for all $ n \geq 35$ and so on.

The following result gives an equivalent criterion for the validity of the inequality;
\begin{equation}\label{gt12}\frac {p^2_{n+1}}{n+1} <\pi\left(p^2_{n+1} \right) \; \; \mbox{for all $n \geq 2.$}\end{equation}
\begin{lemma}\label{l03}
Let $n \geq 2$ be a fixed integer, let $ p_n $ denote the $n^{\rm th}$ prime number and for each $s, \; 1 \leq s \leq n,$ let $m^s_r$ denote the $r^{\rm th}$ multiple in the ordered sequence of all products,
  $p_s \prod_{i\geq 1}{p_{s_i}},$ where $p_{s_i}$ are primes not less than $p_s.$ For each $s,$ let $\overline{ \{{m^s_r}\}}$ denote the terms of the sequence $\{{m^s_r}\}_{r \geq 1}$  which are less than
  $p^2_{n+1}.$
  Then  $ \frac{p^2_{n+1}}{n+1} < \pi\left(p^2_{n+1} \right) $ if and only if $|\bigcup_{s=1}^n\overline{ \{{m^s_r}\}}| <   \frac{n( p^2_{n+1})}{n+1} - 2 .$
\end{lemma}
\begin{proof} Note that $ \frac{1}{n+1} = 1 - \frac{n}{n+1} = 1 - \sum_{s=1}^n \frac{1 }{s(s+1)}.$ Therefore $ \frac{p^2_{n+1}}{n+1} < \pi\left(p^2_{n+1} \right) $ if and only if \\
$p^2_{n+1} - \sum_{s=1}^n \frac{p^2_{n+1} }{s(s+1)} < p^2_{n+1} - 2 - |\bigcup_{s=1}^n\overline{ \{{m^s_r}\}}|,$ that is, if and only if  $|\bigcup_{s=1}^n\overline{ \{{m^s_r}\}}| <   \frac{n( p^2_{n+1})}{n+1} -2 ,$
 since \\
 $\sum_{s=1}^n\frac{p^2_{n+1} }{s(s+1)} = \frac{n( p^2_{n+1})}{n+1} .$
\end{proof}

  For each integer $s\geq 1,$ let $p_s$ denote the $s^{\rm th}$ prime number. Let $n, k $ with $n \geq k$ be a pair of integers. For each multiple of $6$ greater than or equal to  $p^2_{n+1}-1,$ that is, $x =6r \geq p^2_{n+1}-1,$ let $S(x,k)$ denote the sum
  \begin{equation}\label{f1} S(x,k):= x + \sum_{j=1}^{k} (-1)^j \left\{ \sum_{1\leq
 s_1 < \cdots < s_{j}  \leq k }  \left[ \frac{x}{\prod_{i=1}^{j} p_{s_i}}
   \right] \right\}.
   \end{equation}
   The sum in Equation (\ref{f1}) is based on the inclusion-exclusion principle and can be considered as a sieve on the sequence of integers;
   \begin{equation}\label{f1s} 1,2,3,4,5, \ldots , x \end{equation}
   which sifts out all integers $y$ for which g.c.d.$(y, p_s) \neq 1$ for some $s, \; 1 \leq s \leq k.$
     Since the expression for the value $S(x,k)$ sifts out the primes $p_j, \; 1 \leq j \leq k,$ and the
      composites $m^s_r,$ defined in Lemma \ref{l03} from the Sequence (\ref{f1s}), the result of the lemma therefore enables us to compare the
      values $S(p^2_{n+1}-1,k)$ with $\frac{p^2_{n+1}-1}{n+1},$ the order of the residue of the sieve $p^2_{n+1}-1 - \sum_{s=1}^n \frac{p^2_{n+1}-1 }{s(s+1)}.$
     Further the comparison can be achieved inductively. In fact the result of the lemma could be extended to any sequence of the form (\ref{f1s}) and enable us to compare $S(x,k)$ with  $\frac{x}{n+1}.$
    Let ${\cal S}(x,k),$ denote the set of all positive integers not exceeding $x$ which are relatively prime to the primes $p_j, \; 1 \leq j \leq k.$ Then $ S(x,k) = |{\cal S}(x,k)|.$ Note that for $x = p^2_{n+1}-1,$ the effect of $S(x,n)$ on the Sequence \ref{f1s} coincides with that of the sieve of Eratosthenes apart from the fact that $S(x,n)$ also sifts out the primes $2,3,5, \ldots, p_n.$ Thus  $\pi\left(p^2_{n+1}\right) =  S(p^2_{n+1}-1,n)+n-1.$

    In the following result we show, in particular, that the number of primes between $p_{n}$ and $p^2_{n+1}$ is unbounded as $n$ increases. The result is an immediate consequence of Theorem \ref{maint32}.

   \begin{corollary}\label{th3} For $n \geq 12,$ let $p_n$ denote the $n^{\rm th}$ prime. Then for each integer $d \geq 1$  there exists an integer $N(d)$ such that
   $\frac {d p^2_{n+1}}{n+1} < S( p^2_{n+1}-1,n),$ for all $ n \geq N(d).$
   \end{corollary}
\begin{proof}
   If $n \geq 12,$ then
   $ \frac {b p^2_{n+1}}{n+1} <\pi\left(p^2_{n+1} \right)$ for some integer $b \geq 2.$
   Since $S( p^2_{n+1}-1,n) = \pi\left(p^2_{n+1} \right ) - n+1$ and
   $ n-1 < \frac {(n+1)^2 \log^2 (n+1)}{n+1} < \frac { p^2_{n+1}}{n+1},$
 we have
 $$  \frac {(b-1) p^2_{n+1}}{n+1} < \frac {b p^2_{n+1}}{n+1} -n+1 < \pi\left(p^2_{n+1} \right) - n+1  = S( p^2_{n+1}-1,n)$$ for all  $n \geq 12.$
  The result of the corollary follows if we put $d=b-1.$
 \end{proof}

 Since $\frac {d (p^2_{n+1}-1)}{n+1} < \frac {d p^2_{n+1}}{n+1}$ we see that if for each integer $n \geq 12,$ we put $d_n = \frac {(n+1)S(p^2_{n+1}-1,n)}{p^2_{n+1}-1}, $ then, from the result of Corollary \ref{th3}, we have an unbounded sequence of rational numbers $\{d_n\}_{n \geq 12}$ such that
$$S(p^2_{n+1}-1,n) = \frac {d_n (p^2_{n+1}-1)}{n+1}.$$ But for each $ n \geq 4,$ $S(p^2_{n+1}-1,n)$ may be computed inductively from $S( p^2_{n+1}-1,3),$ forming a finite sequence of values $S( p^2_{n+1}-1,k),$ $3 \leq k \leq n.$ For each $n \geq 4 $ and $k, \; 3 \leq k \leq n,$ we have \\
 $S(p^2_{n+1}-1,k+1) = S(p^2_{n+1}-1,k) - T( p^2_{n+1}-1,k+1),$
     where
     \begin{equation}\label{fse16}T( p^2_{n+1}-1,k+1) :=  \left[ \frac{ p^2_{n+1}-1}{p_{k+1}} \right] + \sum_{j=1}^{k} (-1)^{j} \left\{ \sum_{1\leq
 s_1 < \cdots < s_{j}  \leq k }  \left[ \frac{ p^2_{n+1}-1}{{p_{k+1}}\prod_{i=1}^{j} p_{s_i}}
   \right] \right\} . \end{equation}
In the same vain, $d_n$ is the last term of a sequence of numbers $\{a_k(n)\}, \; 3 \leq k \leq n,$ defined, for each fixed integer $n \geq k,$ by $a_k(n) :=  \frac {(k+1)S(p^2_{n+1}-1,k)}{p^2_{n+1}-1}.$

We now show that $a_3(n)>1$ for all $n \geq 3.$
For each $n \geq 3,$ put $k_n:=  \frac { p^2_{n+1}-1}{6}.$ Then ${\cal S}( p^2_{n+1}-1,2),$ the residue set
of the sieve that yields the value $S( p^2_{n+1}-1,2),$ consists of terms in the sequence:
\begin{equation}\label{fse9}1, 5,7,11,13, \ldots ,6t-1, 6t+1, \ldots , 6k_n-1  ,\end{equation}
while the residue set of the sieve $(p^2_{n+1}-1) - \sum_{s=1}^2 \frac{p^2_{n+1}-1 }{s(s+1)} = \frac {p^2_{n+1}-1}{3}$ partitions $p^2_{n+1}-1$ into $\frac {p^2_{n+1}-1}{3}$ equal parts.
We show that $\frac { p^2_{n+1}-1}{4} < S( p^2_{n+1}-1,3)$ for all $ n \geq 3.$
But  $$ \frac{p^2_{n+1}-1}{4} = \frac { p^2_{n+1}-1}{3} - \frac { p^2_{n+1}-1}{3\cdot4}$$ while
$$S( p^2_{n+1}-1,3)  = \frac { p^2_{n+1}-1}{3} - \left[ \frac {p^2_{n+1}-1}{5} \right ]+ \left[ \frac { p^2_{n+1}-1}{2 \cdot 5} \right ] + \left[ \frac { p^2_{n+1}-1}{3 \cdot 5} \right ]- \left[ \frac { p^2_{n+1}-1}{2\cdot3 \cdot 5} \right ] .$$ $S( p^2_{n+1}-1,3)$ may, equivalently, be obtained by forming all the products
\begin{equation}\label{fse11}5\cdot1, \; 5\cdot5,5\cdot7,\; 5\cdot11, \; 5\cdot13, \; \ldots , \; 5\cdot(6t-1), \; 5\cdot(6t+1), \; \ldots , \; 5\cdot(6k_n-1)  \end{equation} and subtracting the total number of terms
 in this sequence consisting of all products of magnitude less than $p^2_{n+1}-1$ from $S( p^2_{n+1}-1,2).$ Since $4<5,$ and the average difference between consecutive elements in the Sequence (\ref{fse9}) is $3$ it follows that we
 must have $\frac { p^2_{n+1}-1}{4} < S( p^2_{n+1}-1,3)$ for all $ n \geq 3.$ Thus $a_3(n)>1$ for all $n \geq 3.$

  Note that for each pair of integers $n,$ $k, \; 3 \leq k \leq n,$ we have:
  \begin{equation}\label{fse12}
  \frac {a_k(n)(p^2_{n+1}-1)}{k+1} = S(p^2_{n+1}-1,k).\end{equation}
  If $k=3,$ then $1<a_3(n)<1.2$ for each value of $n \geq 3.$  We know that
    \begin{equation}\label{fse14} \frac {d (p^2_{n+1}-1)}{n+1} <  S( p^2_{n+1}-1,n) \end{equation}
     for some integer $d \geq 2$ if $ n \geq 23.$ As observed above, each value $S( p^2_{n+1}-1,n)$ may be obtained inductively from $S( p^2_{n+1}-1,3)$  and likewise each value $\frac { p^2_{n+1}-1}{n+1}$ may simultaneously be obtained inductively from $\frac{ p^2_{n+1}-1}{3}.$  Corollary \ref{th3} shows that for each fixed integer $n \geq 12,$ we have $a_3(n) \leq a_n(n)=d_n.$

In the example below, we compute $a_{30}(n)$ in some cases when and $n \geq 30.$

\begin{example}\label{ex1} {\rm Let $k =30.$  The table below shows the values
 $a_{30}(n)$ for some  integers $n \geq 30.$ We consider values $n$ for which $p_{31}^2-1 =16128 \leq p^2_{n+1}-1 < \prod_{s=1}^{30}p_s.$ For the last value, $\prod_{s=1}^{30}p_s$ in the table,  we use the formula:
$ \frac{(31)\prod_{s=1}^{30}(p_s-1)}{\prod_{s=1}^{30}p_s}$ to compute the corresponding quotient.
\vspace{4ex}
 \begin{table}[ht]
    \centering
      \begin{tabular}{|c|c|c|} \hline
       $n$ &   $ p^2_{n+1}-1$  & $a_{30}(n)$  \\ \hline
      $30$ &    $16 \; 128$      &  $3.5521 $     \\ \hline
      $1026 $  &  $66\; 896 \; 040$     &  $3.5566$    \\ \hline
      $2052 $   &  $320 \; 803\; 920$     &  $3.5573$      \\ \hline
      $3078  $   &  $799 \; 701 \; 840$     &  $3.5578$       \\ \hline
      $ 4103 $    &  $1\;  518 \; 738 \; 840  $     &   $3.5579    $    \\ \hline
      $5130  $     &  $2\; 499 \; 100 \; 080  $    &  $3.5580  $    \\ \hline
       $ 6156  $   &  $3\; 738 \; 221 \; 880  $    &  $ 3.5580 $    \\ \hline
        $7182  $    &   $ 5\; 277 \; 586 \; 608$    &  $3.5580   $    \\ \hline
        $  7695 $    &  $6\; 150\; 794 \; 328  $    &  $3.5580 $    \\ \hline
         $ 8469  $    &  $7\; 607 \; 851 \; 728   $    &  $3.5579 $    \\ \hline
         $ 9593   $   & $ 10 \; 003 \; 800 \; 360 $   &  $ 3.5579   $    \\ \hline
         $     $  &  $\prod_{s=1}^{30}p_s $    &  $ 3.5579 $    \\ \hline
              \end{tabular}
 \end{table}
   }
\end{example}

 The following is our main observation in this section:

\begin{lemma}\label{th30} Let $n,$ $k$ with, $ k \leq n,$ be a pair of integers. Then
$2T(p^2_{n+1}-1,k+1) < \frac {a_{k}(n)( p^2_{n+1}-1)}{(k+1)(k+2)}$ for all $k \geq 150.$
\end{lemma}
\begin{proof} We first
%
   show that
 $\{a_{k}(n)\}$ is a nondecreasing sequence for each $n.$
 To get a more explicit estimate for $a_{k}(n)$ for a given value of $k \geq 30,$ we note that if $k=n,$ then
  \begin{eqnarray}\label{fv9}
 a_{n}(n)&= &\frac {(n+1)S(p^2_{n+1}-1,n)}{p^2_{n+1}-1} \nonumber \\[2ex]
 \hspace{4ex} & = & \frac {(n+1)(\pi\left(p^2_{n+1} \right ) - (n-1) )}{p^2_{n+1}-1} \nonumber \\[2ex]
 \hspace{4ex} & > & \frac {(n+1)(\frac{ p_{n+1}^2 }{(2{\log}(p_{n+1}))} - (n-1) )}{p^2_{n+1}-1} \nonumber \\[2ex]
 \hspace{4ex} & > & \frac {(n+1)}{(2{\log}(p_{n+1}))} - \frac{(n+1) (n-1)}{(n+1)^2 ({\log}(n+1))^2} \; \; \;  > \; \;  \frac {(n+1)}{(2{\log}(p_{n+1}))} - \frac{1}{10}. \nonumber
 \end{eqnarray}
 By Theorem \ref{ross1}, $\frac {m}{({\log}(\frac{m}{1.64}))} < \frac {m}{({\log}(m-\frac{1}{2}))} < \pi (m)$ for $m \geq 67.$ We see therefore that we must have \\
 $a_{n}(n) > \frac {(n+1)}{(2{\log}(p_{n+1}))}$ for all $n \geq 30.$

 Recall that ${\cal S}(p^2_{n+1}-1,k)$ represents the residue set of the sieve of Equation \ref{f1}, that is, is the set of all positive integers not exceeding $p^2_{n+1}-1$ which are relatively prime to
 the primes $p_s, \; 1 \leq s \leq k.$ Let $ \{q^{k}_r\}, \; r \geq 1$ be the sequence of elements of ${\cal S}(p^2_{n+1}-1 ,k),$ so that $q^{k}_1 = 1, q^{k}_2= p_{k+1}, q^{k}_3 = p_{k+2}, \ldots .$ Then
  $q^{k}_r$ is a prime whenever $q^{k}_r < p^2_{k+1}.$ For $n > k,$  ${\cal S}(p^2_{n+1}-1,k+1)$  is obtained from ${\cal S}(p^2_{n+1}-1, k)$ by sifting out all products $q^{k}_r p_{k+1}$ less than $p^2_{n+1}-1,$ where, for each $r,$ $q^{k}_r$ is an element of ${\cal S}(p^2_{n+1}-1, k)$ or, equivalently,
   $$S(p^2_{n+1}-1,k+1) = S(p^2_{n+1}-1,k) - T( p^2_{n+1}-1,k+1).$$
  Now
  $$\frac {a_k(n)(p^2_{n+1}-1)}{k+1} = S(p^2_{n+1}-1,k)$$ and
  \begin{eqnarray}\label{fce9}
 \frac {a_{k}(n)(p^2_{n+1}-1)}{k+2}& = &\frac {a_{k}(n)(p^2_{n+1}-1)}{k+1} - \frac{a_{k}(n)(p^2_{n+1}-1) }{(k+1)(k+2)} \nonumber \\[2ex]
\hspace{4ex} & =  &\frac {p^2_{n+1}-1}{\frac{1}{a_{k}(n)}(k+1)} - \frac{p^2_{n+1}-1 }{\frac{1}{a_{k}(n)} (k+1)(k+2)} \nonumber
 \end{eqnarray}
 Thus $a_{k+1}(n) = a_k(n)$ if $$T( p^2_{n+1}-1,k+1) = \frac{p^2_{n+1}-1 }{\frac{1}{a_{k}(n)} (k+1)(k+2)}.$$ It follows that $a_{k+1}(n) \geq a_k(n)$ only if
 $$T( p^2_{n+1}-1,k+1) \leq \frac{p^2_{n+1}-1 }{\frac{1}{a_{k}(n)} (k+1)(k+2)}.$$
  From our remarks above, it suffices to show that
 \begin{equation}\label{fs2}
 r(\frac{1}{a_k(n)}(k+1)(k+2))< p_{k+1}q^{k}_r
  \end{equation}
 for each $r \geq 1$ for which both products are less than $p^2_{n+1}-1.$ Since $(k+1)(\log (k+1) ) < p_{k+1},$  it suffices to show that $r(\frac{1}{a_k(n)}(k+2)) < (\log (k+1) )q^{k}_r$ or, equivalently,
 $\frac{r}{\log (k+1)}(\frac{1}{a_k(n)}(k+2)) < q^{k}_r.$ If $1<q^{k}_r <  p^2_{k+1},$ then
 $ q^{k}_r$ is equal to a prime number $p_{s}$ with $s >k.$ We know that $s\log s < p_s.$ Treating $s\log s$ as a function of $s$ we get its derivative to be $1+\log s.$ Treating
  $\frac{r}{\log (k+1)}(\frac{1}{a_k(n)}(k+2))$ as a function of $r$ we get its derivative to be equal to $\frac{1}{\log (k+1)}(\frac{1}{a_k(n)}(k+2)).$ By virtue of our estimation of $a_n(n)$ above we assume
   that $a_k(n) \geq \frac {(k+1)}{(2{\log}(p_{k+1}))}$ or that $ \frac {(k+1)}{(2{\log}(p_{k+1}))}$ is a close estimate for $a_k(n).$
   Then we would have that
  $\frac{1}{\log (k+1)}(\frac{1}{a_k(n)}(k+2))$ is less than or approximately equal to  $\frac {(k+2)(2{\log}(p_{k+1}))}{(k+1)(\log (k+1) )}.$
  Now
   \begin{eqnarray}\label{fce9}
 \frac {(k+2)(2{\log}(p_{k+1}))}{(k+1)(\log (k+1))}& = &\frac {2{\log}(p_{k+1})}{\log (k+1) } + \frac {2{\log}(p_{k+1})}{(k+1)\log (k+1) } \nonumber \\[2ex]
\hspace{4ex} & <  &\frac {2{\log}((k+1)(\log (k+1)+ \log \log (k+1) ))}{\log (k+1) } + 0.1 \nonumber \\[2ex]
\hspace{4ex} & =  & 2 + \frac {2{\log}(\log (k+1)+ \log \log (k+1) )}{\log (k+1) } + 0.1  \; \; \; \; < 3 \nonumber
 \end{eqnarray}
 But
  $3 < 1+\log s, \; \; s \geq k+1$ and $k \geq 30.$ This establishes the Inequality (\ref{fs2}) for
   $1\leq q^{k}_r <  p^2_{k+1}.$ For $k \geq 30,$ the set ${\cal S}(p^2_{n+1}-1,k)$ is more dense over the interval $1\leq q^{k}_r <  p^2_{k+1},$ than over the interval
   $q^{k}_r \geq  p^2_{k+1}.$ The above argument therefore suffices for the cases $q^{k}_r \geq  p^2_{k+1}.$
   For $s \geq 150,$ $6 < 1+\log s$ and
     this completes the proof of the lemma.
  \end{proof}

A famous conjecture of Hardy and Littlewood states that
\begin{equation}\label{rgt10}\pi_2\left(n \right)\approx 2C_2 \int_2^n \frac{ dt}{\log^2 t} = L_2(n) \end{equation}
where $C_2 = 0.661618155....$
 $L_2(n)$ may be estimated by \begin{equation}\label{rgt9}
  2C_2\frac{n}{\log^2 n} . \end{equation}
Later we shall show that the Hardy-Littlewood conjecture implies Conjecture \ref{maint4}.

 We note that the $n^{\rm th}$ Catalan number,  ${\frac{1}{n+1}} {\binom{2n}{n}},$ may be expressed in the form
   \begin{eqnarray}\label{cat1}
{\frac{1}{n+1}} {\binom{2n}{n}}& =   &  \left (1 - {\frac{n}{n+1}} \right) {\binom{2n}{n}}  \nonumber \\[2ex]
\hspace{4ex} & =  & {\binom{2n}{n}} -  {\binom{2n}{n+1}}    \nonumber \\[2ex]
  \hspace{4ex} & =  &  {\binom{2n}{n}} -\sum_{i=1}^{n}{2n-i \choose n}.
 \end{eqnarray}
   We can therefore consider ${\frac{1}{n+1}} {\binom{2n}{n}}$ as representing the residue of a sieve defined on the sequence of consecutive integers from $1$ to  ${\binom{2n}{n}}$ in the following way. Sift out
    ${\binom{2n-1}{n}}$  numbers from the sequence in an evenly distributed fashion. Since ${\binom{2n-1}{n}} = {\frac{1}{2}} {\binom{2n}{n}},$ this is equivalent to striking out all the even numbers.  From the remaining integers, ${2n-2 \choose n}$ numbers in the sequence are sifted out in an evenly distributed fashion and so on for each of the integers $i, \; 3 \leq i \leq n.$  In any case
   the ${\frac{1}{n+1}} {\binom{2n}{n}}$ elements of the residue set are assumed to be evenly distributed over the interval $1$ to ${2n \choose n}.$ Thus any two consecutive elements in the
    residue of the sieve may be assumed to be separated by an interval of $ n+1.$ So $n+1$ is the average density of the residue of the sieve of Equation (\ref{cat1}).

  Let ${C^{2n}_n}= {\binom{2n}{n}}.$   Writing
    \begin{eqnarray}\label{sc1}
 \frac {p^2_{n+1}}{n+1} & = &    \frac {p^2_{n+1}}{(n+1)}\left(\frac {1}{C^{2n}_n} \binom{2n}{n} \right)  \nonumber \\[2ex]
\hspace{4ex} & =  &   \frac {p^2_{n+1}}{C^{2n}_n} \left({\binom{2n}{n}} -\sum_{i=1}^{n}{2n-i \choose n} \right), \nonumber
 \end{eqnarray}
we see that we can treat  $\left[\frac{p^2_{n+1}}{n+1} \right]$ as a residue of the sieve of Equation (\ref{cat1}) when restricted to the sequence of integers $1, \ldots ,{p^2_{n+1}}.$

 \section{Proof of Theorem \ref{maint33}}\label{proof2}

For $n \geq 3,$ let $x = 6r \geq p^2_{n+1}-1$ be an integer. For $k \leq n,$ the equation
  \begin{eqnarray}\label{sx1} x - \sum_{s=1}^k \frac{x }{s(s+1)} = \frac {x}{k+1}  \end{eqnarray}
  can be viewed as a sieve on the magnitude $x$ with residue $\frac {x}{k+1}.$
  When $k=2,$ then the residue set of the Sieve (\ref{sx1}) on the sequence $1,2,3, \;  \ldots, \; x$  has order
  $\frac {x}{3}.$  On the other hand $S(x,2) = \frac {x}{3}$ is the order of the residual set after sifting out all the multiples of $2$ and $3$ not exceeding $x.$ Let $W$ and $S$ denote the respective residue sets of order   $\frac {x}{3}.$ Upon multiplying the Sieve (\ref{sx1}) by $\frac{1}{2}$ we obtain a sieve which can be considered as sifting out the elements of $W$ in pairs. We know that the sieve of Equation (\ref{f1}) may be extended naturally to one on pairs of the form  $(6t-1, 6t+1)$ in $S.$
In this section we shall compare the order of the magnitude of the residue sets of the respective extensions of the two sieves on $W$ and $S.$
   In particular if $k=n,$ the residue of the extension of the sieve of Equation (\ref{f1}), when $x = p^2_{n+1}-1,$ will consist of twin primes.
  In this event we shall see that as $k$ increases the extension of the Sieve (\ref{sx1}) on $W$ is dominant or more porous than that of the Sieve (\ref{f1}) on $S.$ The residue of the extension of the Sieve (\ref{sx1}) will then be seen to be unbounded as $k$ increases and this, in turn, will establish  Theorem \ref{maint33}.

We first prove the following preliminary result. After sifting out the multiples of $2$ and $3$ from the sequence $1,2,3, \ldots , (p^2_{n+1}-1)$ the residue set is of the form  $\{1, (p^2_{n+1}-2)\} \cup \{(6t-1, 6t+1) \; |\; 1 \leq t < \frac{(p^2_{n+1}-1) }{6} \}.$ The sieve of Eratosthenes extends naturally to the set $\{(6t-1, 6t+1) \; |\; 1 \leq t < \frac{p^2_{n+1}-1 }{6} \}$ and for each integer $n \geq 3,$ we let
 ${\cal S}(p^2_{n+1})$ denote the set
 $${\cal S}(p^2_{n+1}) = \{t \in \mathbb{N} \; | \;  1 \leq t < \frac{(p^2_{n+1}-1) }{6} , \;\; 6t-1 \; \mbox{or} \; 6t+1 \; \mbox{is a composite}\}.$$
 Note that $6t-1$ or $6t+1$ is a composite if and only if either is equal to $m_r^s$ for some integers $r,s.$ Twin primes are of the form
 $(6t-1, 6t+1)$ for some integer $t \geq 1$ apart from the pair $(3, 5).$ Thus the order of the complement $|\left({\cal S}(p^2_{n+1})\right)'| =  \pi_2\left(p^2_{n+1} \right) - 1.$
 We have the following:
\begin{lemma}\label{maint31}  For $n \geq 3,$ let $ p_n $ denote the $n^{\rm th}$ prime.  Then
 $$ \frac{p^2_{n+1}}{2(n+1)}  < \pi_2\left(p^2_{n+1} \right) $$ if and only if
$$  |{\cal S}(p^2_{n+1})|  < \frac{(n-2)p^2_{n+1}}{6(n+1)} + \frac{1}{6} .$$
\end{lemma}
\begin{proof}
 Let $n \geq 3$ be an integer. We recall from the result of Lemma \ref{l03} that if for each $s, \; 1 \leq s \leq n,$ we let $m^s_r$ denote the $r^{\rm th}$ multiple in the ordered sequence of all products,
  $p_s \prod_{i\geq 1}{p_{s_i}},$ where $p_{s_i}$ are primes not less than $p_s$ and for each $s,$ let $\overline{ \{{m^s_r}\}}$ denote the terms of the sequence $\{{m^s_r}\}_{r \geq 1}$  which are less than
  $p^2_{n+1},$
  then  $ \frac{p^2_{n+1}}{n+1}  < \pi\left(p^2_{n+1} \right) $ if and only if $|\bigcup_{s=1}^n\overline{ \{{m^s_r}\}}| <  \sum_{s=1}^n \frac{p^2_{n+1} }{s(s+1)}-2.$
 Therefore for all values of $n \geq 1,$
 $$|\overline{ \{{m^1_r}\}}|+|\overline{ \{{m^2_r}\}}| = \left[\frac{p^2_{n+1} }{2}\right] + \left[\frac{p^2_{n+1} }{6}\right] -2 =  \sum_{s=1}^2\left[\frac{p^2_{n+1} }{s(s+1)}\right]  -2.$$
 $\overline{ \{{m^1_r}\}}$ and  $\overline{ \{{m^2_r}\}}$ are obtained by sifting out the multiples of $2$ and $3$ in the sequence $1, \ldots , (p^2_{n+1}-1).$ Thus the residue set consists of
 $\{1, (p^2_{n+1}-2)\} \cup \{(6t-1, 6t+1) \; |\; 1 \leq t < \frac{(p^2_{n+1}-1) }{6} \}.$ On the other hand  $\sum_{s=1}^2\left[\frac{p^2_{n+1} }{s(s+1)}\right]$ sifts out integers in the sequence $1, \ldots , (p^2_{n+1}-1)$ to leave any two consecutive elements in the
    residue of the sieve separated by an interval of $ 3.$ Thus we can assume that the residue consists of
 $ \{1, (p^2_{n+1}-3)\} \cup \{(6t-2, 6t+1) \; |\; 1 \leq t < \frac{(p^2_{n+1}-1) }{6}  \}.$
 Now extending the sieve of Eratosthenes to the set $$ T =  \{(6t-1, 6t+1) \; |\; 1 \leq t < \frac{(p^2_{n+1}-1) }{6}  \}$$ leaves a residue set of order $\pi_2\left(p^2_{n+1} \right) -1.$
  On the other hand we have
 $\frac{1}{6} - \sum_{s=3}^n \frac{1}{2s(s+1)} = \frac{1}{6} - \frac{(n-2)}{6(n+1)} = \frac{1}{2(n+1)}.$  Thus we see that
  $ \frac{p^2_{n+1}}{2(n+1)}  < \pi_2\left(p^2_{n+1} \right) $ if and only if
  $\frac{p^2_{n+1}}{6} - \sum_{s=3}^n \frac{p^2_{n+1} }{2s(s+1)} < \frac{(p^2_{n+1}-1) }{6} - |{\cal S}(p^2_{n+1})|,$ that is,
 if and only if
 $$  |{\cal S}(p^2_{n+1})|  <   \frac{(n-2)p^2_{n+1}}{6(n+1)} + \frac{1}{6} $$ since
 $\sum_{s=3}^n\frac{p^2_{n+1} }{2s(s+1)} =  \frac{(n-2)p^2_{n+1}}{6(n+1)}$ and $\frac{p^2_{n+1}}{6}=\frac{(p^2_{n+1}-1) }{6} + \frac{1}{6}.$
 \end{proof}

  For each integer $s\geq 1,$ let $p_s$ denote the $s^{\rm th}$ prime number. For $k \geq 3$ let $x_k:= 5\cdot 7\cdot 11\cdot \ldots \cdot p_{k}.$ For each integer $r \geq 1,$ consider the function
\begin{eqnarray}\label{tw2}
 \begin{array}{ll}
 \phi^k_2(r \cdot x_k)&=r \cdot x_k(1-\frac{2}{p_{3}})(1-\frac{2}{p_{4}})\dots (1-\frac{2}{p_{k}})\\[2ex]
        &=r \cdot x_k + \sum_{j=3}^{k} (-1)^j \left\{ \sum_{3\leq
 s_1 < \cdots < s_{j}  \leq k}  \left( \frac{2^{j-2}r \cdot x_k}{\prod_{i=1}^{j} p_{s_i}}
   \right) \right\}.
 \end{array}
 \end{eqnarray}
 Then $\phi^k_2$ enumerates the pairs $(6t-1, 6t+1), \; 1 \leq t \leq r \cdot x_k,$ both of which are relatively prime to $x_k.$ The pairs $(6t-1,6t+1)$ in
the residue of the sieve $\phi^k_2$ which are both less than $p^2_{k+1}$ are twin primes. Our aim is to show that the number of twin primes bounded by $p_n$ and
$p^2_{n+1}$  in unbounded as $n$ increases.

 For each pair of integers $ n, \; k$ with, $3 \leq k \leq n,$ let $6k_n =   p^2_{n+1}-1$ and
  ${\cal R}(p^2_{n+1}-1,k)$ denote the set
 \begin{equation}\label{f8} {\cal R}(p^2_{n+1}-1,k):= \{1, 6k_n-1\} \cup \{t \in \mathbb{N} \; | \;  1 \leq t < k_n , \;\; \mbox{both} \; 6t-1 \; \mbox{and} \; 6t+1 \; \mbox{are not divible by} \; \; p_s,\;\; 3 \leq s \leq k\}. \nonumber
   \end{equation}
Going back to our definition of $ S(x,k)$ we note that
        after sifting out the multiples of $2$ and $3,$ the sieve
    $$ S(p^2_{n+1}-1,k)= p^2_{n+1}-1 + \sum_{j=1}^{k} (-1)^j \left\{ \sum_{1\leq
 s_1 < \cdots < s_{j}  \leq k }  \left[ \frac{p^2_{n+1}-1}{\prod_{t=1}^{j} p_{s_t}}
   \right] \right\}$$
   on the sequence of integers
   \begin{equation}\label{ft12} 1,2,3,4,5, \ldots ,    p^2_{n+1}-1 \end{equation}
   extends naturally to a sieve on the sequence of integers
    \begin{equation}\label{ff9}1, 5,7,11,13, \ldots ,6t-1, 6t+1, \ldots , 6k_n-1  .\end{equation}
    On the sequence (\ref{ff9}) we have
    \begin{equation}\label{fs9} S(p^2_{n+1}-1,k)= \frac{p^2_{n+i}-1}{3} - \left\{ \sum_{r=3}^{k}\left(\left[\frac{p^2_{n+1}-1}{p_r}\right] + \sum_{j=1}^{r-1} (-1)^j \left\{ \sum_{1\leq
 s_1 < \cdots < s_{j}  \leq r-1 }  \left[ \frac{p^2_{n+1}-1}{\prod_{i=3}^{j}p_r p_{s_i}}
   \right] \right\} \right )\right \}.
   \end{equation}

    Arranging the terms of the Sequence (\ref{ff9}) as pairs in the set
        \begin{equation}\label{fr9} S:= \{1, 6k_n-1 \} \cup \{(6t-1, 6t+1) \; |\; 1 \leq t < k_n \},\end{equation}
         then Equation (\ref{fs9}) extends naturally to a sieve on the set $S,$ which sifts out pairs $(6t-1, 6t+1)$ whenever $6t+1$ or $6t-1$ is divisible by some prime $p_s, \; 3 \leq s \leq k.$ This practical extension of the sieve (\ref{fs9}) is easily seen to coincide with the effect of the sieve (\ref{tw2}) when truncated at  $p^2_{n+1}-1$ or over the interval  $1 \leq t < k_n <rx_k.$ Thus in both cases
           the resulting residue
          set is $ {\cal R}(p^2_{n+1}-1,k),$ unless $6k_n-1$ is divisible by $p_s, \; \; 3 \leq s \leq k,$ in which case this value is excluded from $ {\cal R}(p^2_{n+1}-1,k).$

   In the same vein if $n>k$ and $6k_n = p^2_{n+1}-1,$ then the effect of
    \begin{equation}\label{fp9} T( p^2_{n+1}-1,k) =  \left[ \frac{ p^2_{n+1}-1}{p_{k+1}} \right] +  \left( \sum_{j=1}^{k} (-1)^{j} \left\{ \sum_{1 \leq
 s_1 < \cdots < s_{j}  \leq k }  \left[ \frac{ p^2_{n+1}-1}{{p_{k+1}}\prod_{i=1}^{j} p_{s_i}}
   \right] \right\}  \right)  \end{equation}
   on the respective Sequence (\ref{ff9}) extends to a sieve on ${\cal R}( p^2_{n+1}-1,k)$
   that sifts out pairs $(6t-1, 6t+1)$ from ${\cal R}( p^2_{n+1}-1,k)$ for which $6t-1$ or $6t+1$ is divisible by $p_{k+1}$ with
   ${\cal R}( p^2_{n+1}-1,k+1)$ as the resulting residue set. Thus the Sieve (\ref{fs9}) extends naturally to a sieve on ${\cal R}( p^2_{n+1}-1,k).$ Let $Q( p^2_{n+1}-1,k+1)$ denote the number of pairs sifted out from ${\cal R}( p^2_{n+1}-1,k)$ in this way. Then $Q( p^2_{n+1}-1,k+1) \leq T( p^2_{n+1}-1,k+1).$

 On the other hand, for   $6k_n = p^2_{n+1}-1,$
 the sieve $ \frac{p^2_{n+1}-1}{3} - \sum_{s=3}^k\frac{p^2_{n+1}-1 }{s(s+1)}$ can be considered as a restriction of the sieve $ p^2_{n+1}-1 + \sum_{s=1}^k\frac{p^2_{n+1}-1}{s(s+1)}$ on the sequence
 $1,2,3,4,5, \ldots ,     p^2_{n+1}-1$
to the sequence of integers
 \begin{equation}\label{ff10}1, 4,7,10,13, \ldots ,   6t-2, 6t+1, \ldots ,6k_n-2. \end{equation}
 Therefore
 \begin{equation}\label{fk10}\frac{p^2_{n+1}-1}{3}- \sum_{s=3}^k\frac{p^2_{n+1}-1 }{s(s+1)} = \frac {p^2_{n+1}-1}{(k+1)}.\end{equation}
  Multiplying Equation (\ref{fk10}) by $\frac{1}{2}$ we obtain
\begin{equation}\label{fk11}\frac{p^2_{n+1}-1}{6} - \sum_{s=3}^k \frac{p^2_{n+1}-1}{2s(s+1)} = \frac{p^2_{n+1}-1}{2(k+1)}\end{equation}
 so that $\frac{p^2_{n+i}-1}{2(k+1)}$ is the order of the residue of the Sequence (\ref{ff10}) when the residue terms are counted in pairs.
  Arranging the terms of the Sequence (\ref{ff10}) as pairs in the set
 \begin{equation}\label{fx11} W:= \{1,6k_n-2\} \cup \{(6t-2, 6t+1) \; |\; 1 \leq t < k_n \}\end{equation}
  we can consider the residue set of the Sequence (\ref{ff10}) as consisting of
   $\frac{p^2_{n+1}-1}{2(k+1)}$ pairs.
 Our aim is to compare $\frac{p^2_{n+1}-1}{2(k+1)}$ with $|{\cal R}(p^2_{n+1}-1,k)|$ as $k$ increases.

 For each integer pair of integers $ n, \; k,$ with, $3 \leq k \leq n,$ let \begin{equation}\label{fx12} a_2(k,n) : =  \frac{2(k+1)|{\cal R}(p^2_{n+1}-1,k)|}{p^2_{n+1}-1}.\end{equation}
 In the example below, we compute  $a_2(k,n)$ for some cases when $k=150$ and $n \geq 150.$
 \begin{example}\label{ex2} {\rm In this case $k =150.$  The table below shows the values of
 $a_2(150,n )$ for some cases when  $n \geq 150,$ that is, $p_{151}^2-1 = 769 \; 128 \leq p^2_{n+1}-1 < \prod_{s=1}^{150}p_s.$

  \begin{table}[ht]
    \centering
          \begin{tabular}{|c|c|c|} \hline
       $n$ &   $ p^2_{n+1}-1$  & $a_2(150,n)$  \\ \hline
     $150$ &    $769 \; 128$      &  $2.5522 $     \\ \hline
      $1026 $  &  $66\; 896 \; 040$     &  $2.7079$    \\ \hline
      $2052 $   &  $320 \; 803\; 920$     &  $2.7545$      \\ \hline
      $3078  $   &  $799 \; 701 \; 840$     &  $2.7519$       \\ \hline
      $ 4103 $    &  $1\;  518 \; 738 \; 840  $     &   $2.7426    $    \\ \hline
      $5130  $     &  $2\; 499 \; 100 \; 080  $    &  $2.7348  $    \\ \hline
       $ 6156  $   &  $3\; 738 \; 221 \; 880  $    &  $ 2.7293 $    \\ \hline
        $7182  $    &   $ 5\; 277 \; 586 \; 608$    &  $2.7252   $    \\ \hline
        $  7695 $    &  $6\; 150\; 794 \; 328  $    &  $2.7235$    \\ \hline
         $ 8469  $    &  $7\; 607 \; 851 \; 728   $    &  $2.7215 $    \\ \hline
         $ 9593   $   & $ 10 \; 003 \; 800 \; 360 $   &  $ 2.7193   $    \\ \hline
              \end{tabular}
 \end{table}
 }
\end{example}
 The following is our main observation:
 \begin{theorem}\label{th5}
  $ \frac{p^2_{n+1}-1}{2(n+1)} < |{\cal R}(p^2_{n+1}-1,n)|$ for all integers $n \geq 150.$
   \end{theorem}
 \begin{proof} Consider the sieve $ x - \sum_{s=1}^k \frac{x }{s(s+1)}$  and that of Equation (\ref{f1}).  When  $x = p^2_{n+1}-1,$ then the result of Lemma \ref{th30} shows that $2T(p^2_{n+1}-1,k+1) < \frac {a_{k}(n)( p^2_{n+1}-1)}{(k+1)(k+2)}$ for all integers $n, k$ with $150 \leq k \leq n.$ The result is obtained by comparing  $$S(p^2_{n+1}-1,k+1) = |{\cal S}(p^2_{n+1}-1,k)| = S(p^2_{n+1}-1,k) - T( p^2_{n+1}-1,k+1)$$
  and
  $\frac {a_{k}(n)(p^2_{n+1}-1)}{k+1} - \frac{a_{k}(n)(p^2_{n+1}-1) }{(k+1)(k+2)}$ when $k \geq 150.$ But
 $${\cal S}(p^2_{n+1}-1,k) = {\cal R}(p^2_{n+1}-1,k) \cup ({\cal R}(p^2_{n+1}-1,k))^c$$
 where $ ({\cal R}(p^2_{n+1}-1,k))^c$ is the complement of ${\cal R}(p^2_{n+1}-1,k)$ in ${\cal S}(p^2_{n+1}-1,k).$ By the result of Lemma \ref{th30} for each integer $r$ for which
$ r(\frac{2}{a_k(n)}(k+1)(k+2)) < p^2_{n+1}-1$ there corresponds a pair $(6t-1, 6t+1) \in {\cal R}(p^2_{n+1}-1,k+1)$ or $(q^{k+1}_{r_1}, q^{k+1}_{r_2}) \in ({\cal R}(p^2_{n+1}-1,k+1))^c$ with at least one component divisible by $p_{k+1}.$ Now for $150 \leq k \leq n$ and $x = p^2_{n+1}-1$ compare the sieve $ \frac{x}{6} - \sum_{s=1}^{150} \frac{x }{2s(s+1)}$ with the restriction of $\phi^k_2$ to $\frac{x}{6}.$ Then the above correspondence would still hold when $k \geq 150$ and in the latter case we would have that for each for each integer $r$ for which
$ r(\frac{2}{a_2(k,n)}(k+1)(k+2)) < p^2_{n+1}-1$ there corresponds a pair $(6t-1, 6t+1) \in {\cal R}(p^2_{n+1}-1,k+1)$  with at least one component divisible by $p_{k+1}.$
 This is equivalent to the statement $Q(p^2_{n+1}-1,k+1) \leq \frac {a_{2}(k,n)( p^2_{n+1}-1)}{2(k+1)(k+2)}$ for all $k \geq 150.$
 Thus if we could show that
\begin{equation}\label{fx13}\frac{p^2_{n+1}-1}{2\cdot151} < |{\cal R}(p^2_{n+1}-1,150)| \end{equation}
 for all $n \geq 150,$ then we would have our required result $ \frac{p^2_{n+1}-1}{2(n+1)} < |{\cal R}(p^2_{n+1}-1,n)|$ for all $n \geq 150.$ Example \ref{ex2} establishes the Inequality (\ref{fx13}) in a few cases. Since $k=150$ is relatively small the residue set may be assumed to be relatively evenly distributed and consequently the Inequality holds for all $n \geq 150.$
  \end{proof}

 As a consequence of Theorem \ref{th5} we have for $x = p^2_{n+1}-1,$ $$\left[\frac{p^2_{n+1}}{2(n+1)} \right] \leq \frac {x}{2(n+1)}  < |{\cal R}(x,n)|+ \pi_2\left(p_{n} \right) = \left \{ \begin{array}{ll}
 \pi_2\left(p^2_{n+1} \right) +1 \;& \; \mbox{if $p_n = 6t-1$ and $6t+1$ is a prime} \\[2ex]
    \pi_2\left(p^2_{n+1} \right) \;& \; \mbox{otherwise},
    \end{array} \right.
    $$
    for all $n \geq 150.$
The cases $20 \leq n \leq 149$ of Theorem \ref{maint33} may be checked independently.

In our proof of Theorem \ref{th5} we saw that for $k \geq 150$ $a_{2}(k,k)$ increases with $k.$ Thus several cases of Conjecture \ref{maint4} may be verified in a similar manner.
 For example, $  \left[\frac{p^2_{n+1}}{(n+1)} \right] < \pi_2\left(p^2_{n+1} \right)$ for all $ n \geq 100,$ $\left[\frac{3p^2_{n+1}}{2(n+1)} \right] < \pi_2\left(p^2_{n+1} \right) $ for all $ n \geq 200,$ $ \left[\frac{2p^2_{n+1}}{(n+1)} \right] < \pi_2\left(p^2_{n+1} \right) $ for all $ n \geq 300$ and so on.

 We now show that the Hardy-Littlewood conjecture (\ref{rgt10}) implies Conjecture \ref{maint4}.
  Accordingly to the Hardy-Littlewood conjecture
   $\pi_2\left(n \right)$ is approximately equal to $ 2C_2\frac{n}{\log^2 n}.$
  Thus to show that
 $$ a\frac{p^2_{n+1}}{2(n+1)}  < \pi_2\left(p^2_{n+1} \right) $$
it would suffice to prove that $$  2C_2\frac{ p_{n+1}^2}{(4{\log}^2(p_{n+1}))}>a\frac{p^2_{n+1}}{2(n+1)}$$ for all $n$ sufficiently large, which is equivalent to
 ${\log}^2 p_{n+1}< C_2\frac{(n+1)}{a}.$ By Corollary \ref{ross4} (ii)
$$p_{n+1}<(n+1)({\log}(n+1)+{\log}{\log}(n+1))$$ for $n>6.$  Thus, it would be enough to show that
 $${\log}^2((n+1)({\log}(n+1)+{\log}{\log}(n+1)))< C_2\frac{(n+1)}{a}.$$
 Regarding  ${\log}^2((n+1)({\log}(n+1)+{\log}{\log}(n+1)))$ as a function of $n,$ we get its derivative
  to be less than
  $$2{\log}((n+1)({\log}(n+1)+{\log}{\log}(n+1)))\left( \frac{3}{(n + 1)} \right).$$
 Noting that $n+1 >  {\log}(n+1)$ and $n+1 > {\log}{\log}(n+1)$ this in turn implies that the derivative is less than
   $$\left( \frac{6}{(n + 1)} \right)(\log 2 + 2\log (n+1))$$
  which is smaller than $\frac{C_2}{a}$ for all $n$ sufficiently large. Thus for values of $n$ large enough  $${\log}^2((n+1)({\log}(n+1)+{\log}{\log}(n+1)))$$
  is less than $C_2\frac{(n+1)}{a}.$ The conjecture, therefore, implies that $\frac{p^2_{n+1}}{2(n+1)}$ is a weak lower bound for $  2C_2\frac{ p_{n+1}^2}{(4{\log}^2(p_{n+1}))}.$
  \begin{acknow}
  We would like to thank the following people for their helpful comments and for identifying errors and oversights in earlier versions of this work; Dang Vo Phuc, Stephan Wagner, Berndt Gensel, Shalin Singh and Abebe Tufa. We are also thankful to P. Kaelo for his assistance with several computer programmes that provided the empirical results in this paper.
  \end{acknow}

 \end{document}